\RequirePackage[l2tabu, orthodox]{nag}

\documentclass[12pt]{amsart}
\usepackage{fullpage,url,amssymb,enumerate,colonequals}
\usepackage[all]{xy} % for commutative diagrams
\usepackage{mathrsfs} % for \mathscr (script letters)

\usepackage[OT2,T1]{fontenc}
\DeclareSymbolFont{cyrletters}{OT2}{wncyr}{m}{n}
\DeclareMathSymbol{\Sha}{\mathalpha}{cyrletters}{"58}

\usepackage{color}

\newcommand{\defi}[1]{\textsf{#1}} % for defined terms

\newcommand{\Aff}{\mathbb{A}}
\newcommand{\C}{\mathbb{C}}
\newcommand{\F}{\mathbb{F}}

\newcommand{\PP}{\mathbb{P}}
\newcommand{\Q}{\mathbb{Q}}

\newcommand{\Qbar}{{\overline{\Q}}}

\newcommand{\kbar}{{\overline{k}}}
\newcommand{\Kbar}{{\overline{K}}}

\newcommand{\Fbar}{{\overline{\F}}}

\DeclareMathOperator{\Char}{char}

\DeclareMathOperator{\Gal}{Gal}

\newcommand{\dyn}{{\operatorname{dyn}}}

\newcommand{\To}{\longrightarrow}
\newcommand{\Union}{\bigcup} % union of a collection

\newtheorem{theorem}{Theorem}[section]
\newtheorem{lemma}[theorem]{Lemma}
\newtheorem{corollary}[theorem]{Corollary}
\newtheorem{proposition}[theorem]{Proposition}

\theoremstyle{definition}

\theoremstyle{remark}
\newtheorem{remark}[theorem]{Remark}

\makeatletter
\g@addto@macro\bfseries{\boldmath} % This makes math in section titles bold.
\makeatother

\usepackage{microtype}

\usepackage[
	pagebackref,
	pdfauthor={Bjorn Poonen}, % add other authors
]{hyperref}
\usepackage[alphabetic,backrefs,lite,nobysame]{amsrefs} % for bibliography

\begin{document}

\title{Gonality of dynatomic curves and \\ strong uniform boundedness of preperiodic points}
\author{John R. Doyle}
\address{Department of Mathematics \& Statistics, Louisiana Tech University, Ruston, LA 71272, USA}
\email{jdoyle@latech.edu}
\urladdr{\url{http://www2.latech.edu/~jdoyle/}}

\author{Bjorn Poonen}
\thanks{The second author was supported in part by National Science Foundation grants DMS-1069236 and DMS-1601946 and Simons Foundation grants \#402472 (to Bjorn Poonen) and \#550033.}
\address{Department of Mathematics, Massachusetts Institute of Technology, Cambridge, MA 02139-4307, USA}
\email{poonen@math.mit.edu}
\urladdr{\url{http://math.mit.edu/~poonen/}}

\date{December 30, 2018}

\begin{abstract}
Fix $d \ge 2$ and a field $k$ such that $\Char k \nmid d$.
Assume that $k$ contains the $d$th roots of $1$.
Then the irreducible components of the curves over $k$
parameterizing preperiodic points of polynomials of the form $z^d+c$
are geometrically irreducible and have gonality tending to $\infty$.
This implies the function field analogue of the 
strong uniform boundedness conjecture for preperiodic points of $z^d+c$.
It also has consequences over number fields:
it implies strong uniform boundedness
for preperiodic points of bounded eventual period,
which in turn reduces the full conjecture for preperiodic points 
to the conjecture for periodic points.
\end{abstract}

\maketitle

%****************************************************************************
\section{Introduction}\label{S:introduction}

\subsection{Dynatomic curves}

Fix an integer $d \ge 2$.
Let $k$ be a field such that $\Char k \nmid d$.
View $f = f_c \colonequals z^d+c$ as a polynomial in $z$ with
coefficients in $k[c]$.
Let $f^n(z)$ be the $n$th iterate of $f$;
in particular, $f^0(z) \colonequals z$.
If $n$ and $m$ are nonnegative integers with $n>m$,
then any irreducible factor of $f^n(z)-f^m(z) \in k[z,c]$
defines an affine curve over $k$.
By a \defi{dynatomic curve} over $k$, 
we mean any such curve, or its smooth projective model.
Any $k$-point on such a curve yields $c_0 \in k$
equipped with a \defi{preperiodic} point in $k$,
that is, an element $z_0 \in k$ that under iteration of $x^d+c_0$
eventually enters a cycle; the length of the cycle is called
the \defi{eventual period}.
We consider two dynatomic curves to be different 
if the corresponding closed subschemes of $\Aff^2_k$ are distinct.
Section~\ref{S:classification of dynatomic curves} describes
all dynatomic curves in characteristic~$0$ explicitly.

\subsection{Gonality}

Let $\kbar$ be an algebraic closure of $k$.
Let $\mu_d = \{x \in \kbar : x^d=1 \}$.

For a curve $X$ over $k$, let $X_{\kbar} = X \times_k \kbar$.
If $X$ is irreducible, define the \defi{gonality} $\gamma(X)$ of $X$
as the least possible degree of a dominant rational map 
$X \dashrightarrow \PP^1_k$.
If $X$ is geometrically irreducible, 
define its \defi{$\kbar$-gonality} as $\gamma(X_{\kbar})$.

The following theorem, and its consequence, 
Theorem~\ref{T:uniform boundedness},
are our main results.

\begin{theorem}
\label{T:gonality}
Fix $d \ge 2$ and $k$ such that $\Char k \nmid d$.
Suppose that $\mu_d \subset k$.
\begin{enumerate}[\upshape (a)]
\item \label{I:geometrically irreducible}
Every dynatomic curve over $k$ is geometrically irreducible.
\item \label{I:gonality of dynatomic curves}
If the dynatomic curves over $k$ are listed in any order,
their gonalities tend to $\infty$.
\end{enumerate}
\end{theorem}

\begin{remark}
Part~\eqref{I:geometrically irreducible} of Theorem~\ref{T:gonality} 
can fail if $\mu_d \not\subset k$.
See Remark~\ref{R:not geometrically irreducible}.
\end{remark}

\begin{remark}
In proving Theorem~\ref{T:gonality} in positive characteristic,
we face the challenge that we do not know explicitly what the 
dynatomic curves are, since we are not sure whether the known factors
of the polynomials $f^n(z)-f^m(z)$ are irreducible.
Some partial results regarding the irreducibility of dynatomic curves 
in positive characteristic may be found in \cite{Doyle-et-al-preprint},
but we will not use them.
Instead, to overcome the difficulties, 
we use a novel argument that is specific to finite fields 
to prove that the degrees of dynatomic curves over the $c$-line
must tend to infinity even though we do not know precisely 
what these curves are;
see Section~\ref{S:gonality in characteristic p}.
\end{remark}

Let us now introduce notation for our next result.
Let $\mu(n)$ denote the M\"obius $\mu$-function.
Then $f^n(z)-z = \prod_{e|n} \Phi_e(z,c)$,
where
\begin{equation}
\label{E:definition of Phi}
	\Phi_n(z,c) \colonequals \prod_{e|n} (f^e(z)-z)^{\mu(n/e)} \in k[z,c].
\end{equation}
Let $Y_1^{\dyn}(n)$ be the curve defined by $\Phi_n(z,c)=0$ in $\Aff^2_k$,
and let $X_1^{\dyn}(n)$ be the normalization of its projective closure.
To simplify notation, we omit the superscript dyn from now on.
General points of $X_1(n)$ parametrize 
polynomials of the form $z^d+c$ equipped with a point of exact order $n$.

The morphism $(z,c) \mapsto (f(z),c)$ restricts to 
an order~$n$ automorphism of $Y_1(n)$, 
so it induces an order~$n$ automorphism $\sigma$ of $X_1(n)$.
The quotient of $X_1(n)$ by the cyclic group generated by $\sigma$
is called $X_0(n)$.
If $\Char k=0$, it is known that $X_1(n)$ is geometrically irreducible,
so $X_0(n)$ is too.

\begin{theorem}
\label{T:X_0}
Fix $d \ge 2$ and a field $k$ of characteristic~$0$.
Then 
\[
	\gamma(X_0(n)) > \left(\frac{1}{2} - \frac{1}{2d} - o(1) \right) n
\]
as $n \to \infty$.
In particular, $\gamma(X_0(n)) \to \infty$.
\end{theorem}

\begin{remark}
Our definition of dynatomic curve does not include quotient curves 
such as $X_0(n)$, 
so the conclusion $\gamma(X_0(n)) \to \infty$ of Theorem~\ref{T:X_0} 
does not follow from Theorem~\ref{T:gonality}\eqref{I:gonality of dynatomic curves}.
In fact, our logic runs in the opposite direction: 
we use Theorem~\ref{T:X_0} in the proof of the characteristic~$0$
case of Theorem~\ref{T:gonality}\eqref{I:gonality of dynatomic curves}.
\end{remark}

\begin{remark}
Although the lower bound in Theorem~\ref{T:X_0} is linear in $n$,
the best upper bound we know, $\gamma(X_0(n)) \le (1+o(1)) d^n/n$
(see Proposition~\ref{P:properties of X_0}\eqref{I:D_0(n)}),
is exponential in $n$.
\end{remark}

To prove Theorem~\ref{T:X_0}, 
we use that $X_0(n)$ already has a morphism to $\PP^1$ of degree
lower than expected for its genus, namely $X_0(n) \stackrel{c}\to \PP^1$.
If it also had a morphism to $\PP^1$ of bounded degree, 
then the Castelnuovo--Severi inequality would make the genus of $X_0(n)$ 
smaller than it actually is, a contradiction.
See Section~\ref{S:gonality in characteristic 0} for details.

To prove Theorem~\ref{T:gonality}\eqref{I:gonality of dynatomic curves},
we use different arguments in characteristic~$0$ and characteristic~$>0$.

In characteristic~$0$, we use that each dynatomic curve dominates
$X_1(n)$ and hence also $X_0(n)$ for some $n$, 
so by Theorem~\ref{T:X_0} its gonality is large when $n$ is large;
this lets us reduce to proving a gonality lower bound for 
the dynatomic curves above $X_1(n)$ for each fixed $n$.
The latter curves for fixed $n$ come in towers 
and we use the Castelnuovo--Severi inequality to work our way up each tower.
See Section~\ref{S:gonality in characteristic 0}.

In characteristic~$p$, we prove that the irreducible components
of $f^n(z)-f^m(z)$ have large degree over the $c$-line,
and we use that to prove that over 
the finite field $\F_q \colonequals \F_p(\mu_d)$ 
their smooth projective models have so many $\F_q$-points over $c=\infty$
that their $\F_q$-gonalities must be large.
Finally, we use a result controlling how gonality of a curve changes
when the base field is enlarged.
See Section~\ref{S:gonality in characteristic p}.

\subsection{Uniform boundedness of preperiodic points}

The growth of gonalities of classical modular curves
implies the strong uniform boundedness theorem for 
torsion points on elliptic curves over function fields
(the function field analogue of Merel's theorem \cite{Merel1996});
see \cite{Nguyen-Saito1996preprint}*{Theorem~0.3}.
Similarly, from Theorem~\ref{T:gonality} we will deduce 
the following function field analogue of a case of 
the Morton--Silverman conjecture \cite{Morton-Silverman1994}*{p.~100}:

\begin{theorem}[Strong uniform boundedness theorem for preperiodic points over function fields]
\label{T:uniform boundedness}
Fix $d \ge 2$ and a field $k$ such that $\Char k \nmid d$.
Let $K$ be the function field of an integral curve over $k$.
Fix a positive integer $D$.
Then there exists $B=B(d,K,D)>0$ such that 
for every field extension $L \supseteq K$ of degree $\le D$ 
and every $c \in L$ not algebraic over $k$,
the number of preperiodic points of $z^d + c$ in $L$
is at most $B$.
If $k$ is finite, the same holds with the words 
``not algebraic over $k$'' deleted.
\end{theorem}

As far as we know, Theorem~\ref{T:uniform boundedness}
is the first theorem proving strong uniform boundedness of preperiodic points
for all members of a nontrivial algebraic family of maps over a global field.
See \cite{Doyle-et-al-preprint}*{Section~4} 
for some related results and arguments.

Another application of Theorem~\ref{T:gonality} is the following, 
which over \emph{number fields} provides a uniform bound
on preperiodic points having a bounded eventual period:

\begin{theorem}\label{T:number fields}
Fix integers $d \ge 2$, $D \ge 1$, and $N \ge 1$. 
Then there exists $B = B(d,D,N) > 0$ such that 
for every number field $K$ satisfying $[K : \Q] \le D$ and every $c \in K$, 
the number of preperiodic points of $z^d + c$ in $K$ 
with eventual period at most $N$ is at most $B$. 
\end{theorem}

Theorems \ref{T:uniform boundedness} and~\ref{T:number fields} 
are proved in Section~\ref{S:uniform boundedness}.
Theorem~\ref{T:number fields} implies that 
the strong uniform boundedness conjecture for \emph{periodic} points
over number fields implies the strong uniform boundedness conjecture 
for \emph{preperiodic} points over number fields, as we now explain:

\begin{corollary}\label{C:preperiodic vs periodic}
Fix integers $d \ge 2$ and $D \ge 1$.
Suppose that there exists a bound $N=N(d,D)$
such that for every number field $K$ satisfying $[K : \Q] \le D$ 
and every $c \in K$, every periodic point of $z^d + c$ in $K$
has period at most $N$.
Then there exists a bound $B' = B'(d,D)$
such that for every number field $K$ satisfying $[K : \Q] \le D$ 
and every $c \in K$, 
the number of preperiodic points of $z^d + c$ in $K$
is at most $B'$.
\end{corollary}

\begin{proof}
By assumption, if $[K:\Q] \le D$ and $c \in K$,
then the preperiodic points of $z^d+c$ in $K$
having eventual period at most $N$ 
are \emph{all} the preperiodic points in $K$.
Therefore the bound $B(d,D,N)$ of Theorem~\ref{T:number fields}
is actually a bound on the \emph{total} number 
of preperiodic points in $K$.
Take $B' = B(d,D,N) = B(d,D,N(d,D))$.
\end{proof}

\section{Classification of dynatomic curves}
\label{S:classification of dynatomic curves}

For $m,n \ge 1$, let $Y_1(m,n)$ be the curve over $k$
whose general points parametrize polynomials $z^d+c$ 
equipped with a preperiodic point that after exactly $m$ steps
enters an $n$-cycle. 
This curve is the zero locus in $\Aff^2_k$ of the polynomial
\[
	\Phi_{m,n}(z,c) \colonequals 
	\frac{\Phi_n(f^m(z),c)}{\Phi_n(f^{m-1}(z),c)}.
\]
For a general point $(z,c) \in Y_1(1,n)$, 
the elements $z$ and $f^n(z)$ are distinct preimages of $f(z)$,
so $z = \zeta f^n(z)$ for some $\zeta \in \mu_d-\{1\}$.
Suppose that $\mu_d \subseteq k$.
For each $\zeta \in \mu_d-\{1\}$, 
let $Y_1(1,n)^{\zeta}$ be the subscheme of $Y_1(1,n)$
defined by the condition $z = \zeta f^n(z)$, so
\[
	Y_1(1,n) = \Union_{\zeta \in \mu_d-\{1\}} Y_1(1,n)^{\zeta}.
\]
Both $(z,c) \mapsto (f(z),c)$ and $(z,c) \mapsto (\zeta^{-1} z,c)$
define isomorphisms $Y_1(1,n)^{\zeta} \to Y_1(n)$.
In particular, $Y_1(1,n)^{\zeta}$ equals 
the curve $\Phi_n(\zeta^{-1}z,c)=0$ in $\Aff^2_k$.
For $m \ge 2$, let $Y_1(m,n)^{\zeta}$ be the inverse image of
$Y_1(1,n)^{\zeta}$ under 
\begin{align*}
	Y_1(m,n) &\To Y_1(1,n) \\
	(z,c) &\longmapsto (f^{m-1}(z),c).
\end{align*}
Then for any $m,n \ge 1$, 
\begin{equation} \label{E:decomposition of Y1mn}
	Y_1(m,n) = \Union_{\zeta \in \mu_d-\{1\}} Y_1(m,n)^{\zeta},
\end{equation}
and $Y_1(m,n)^{\zeta}$ equals the curve $\Phi_n(\zeta^{-1}f^{m-1}(z),c)=0$
in $\Aff^2_k$.
The decomposition~\eqref{E:decomposition of Y1mn} corresponds to 
a factorization
\begin{equation}\label{E:factorization}
	\Phi_{m,n}(z,c) = \prod_{\zeta \in \mu_d - \{1\}} \Phi_n(\zeta^{-1}f^{m-1}(z),c).
\end{equation}

The following is a collection of results from \cites{Bousch1992, Lau-Schleicher1994, Morton1996, Gao2016}.

\begin{theorem}
\label{T:classification of dynatomic curves}
Let $k$ be a field of characteristic~$0$ such that $\mu_d \subset k$.
Then the curves $Y_1(n)$ for $n \ge 1$
and the curves $Y_1(m,n)^\zeta$ for $m,n \ge 1$ and $\zeta \in \mu_d-\{1\}$
are irreducible,
so they are all the dynatomic curves over $k$.
\end{theorem}

Computer experiments (at least for $d=2$) suggest that
Theorem~\ref{T:classification of dynatomic curves} holds 
for any field $k$ such that $\Char k \nmid d$ and $\mu_d \subset k$,
but in positive characteristic this remains unproved. 
See \cite{Doyle-et-al-preprint}, especially Theorems B and~D, 
for some progress in this direction.

\begin{remark}
\label{R:not geometrically irreducible}
If in Theorem~\ref{T:classification of dynatomic curves} 
we drop the hypothesis that $\mu_d \subset k$,
then the irreducible components of $Y_1(m,n)$ over $k$ 
are in bijection with the $\Gal(k(\mu_d)/k)$-orbits in $\mu_d-\{1\}$.
An irreducible component corresponding to an orbit of size greater than $1$
is not geometrically irreducible.
\end{remark}

\section{Gonality in characteristic~0}
\label{S:gonality in characteristic 0}

Given a geometrically irreducible curve $X$ over $k$,
let $g(X)$ denote the genus of its smooth projective model.
Let $D_0(n)$ be the degree of the morphism $X_0(n) \stackrel{c}\To \PP^1$.
Similarly, let $D_1(n) = \deg\left(X_1(n) \stackrel{c}\To \PP^1\right)$.

\begin{proposition}
\label{P:properties of X_0}
Fix $d \ge 2$, and fix a field $k$ of characteristic~$0$.
\begin{enumerate}[\upshape (a)]
\item\label{I:D_1(n)} 
We have $D_1(n) = (1+o(1))d^n$ as $n \to \infty$.
\item\label{I:D_0(n)} 
We have $D_0(n) = (1+o(1))d^n/n$ as $n \to \infty$.
\item\label{I:genus of X_0}
We have $g(X_0(n)) > \left(\frac{1}{2} - \frac{1}{2d} - o(1)\right) d^n$ 
as $n \to \infty$.
\item\label{I:Galois group of X_0}
The Galois group of (the Galois closure of)
the covering $X_0(n) \stackrel{c}\To \PP^1$
is the full symmetric group $S_{D_0(n)}$.
\end{enumerate}
\end{proposition}

\begin{proof}
We may assume $k=\C$.
\begin{enumerate}[\upshape (a)]
\item 
See the proof of~\cite{Morton1996}*{Theorem~13(d)}.
The estimate is deduced from $D_1(n) = \deg_z \Phi_n = \sum_{e|n} \mu(n/e) d^e$,
which follows from~\eqref{E:definition of Phi}.
\item 
The first morphism in the tower $X_1(n) \to X_0(n) \stackrel{c}\To \PP^1$ 
has degree $n$, so $D_0(n) = D_1(n)/n$.
Substitute \eqref{I:D_1(n)} into this.
\item 
More precisely, 
$g(X_0(n)) > \left(\frac{1}{2} - \frac{1}{2d} - \frac{1}{n}\right)
 d^n + O(nd^{n/2})$ as $n \to \infty$, by \cite{Morton1996}*{Theorem~13(d)}.
\item 
This is a consequence of work of Bousch~\cite{Bousch1992} for $d=2$,
and Lau and Schleicher~\cite{Lau-Schleicher1994} for $d>2$.
See also \cite{Morton1998}*{Theorem~B} and \cite{Schleicher2017}.
\qedhere
\end{enumerate}
\end{proof}

A well-known strategy for obtaining lower bounds on gonality
(cf.\ \cite{Nguyen-Saito1996preprint} and \cite{Poonen2007-gonality}*{\S2})
involves the following,
which will tell us roughly that 
if a high genus curve has a relatively low degree map to a low genus curve,
it cannot have a second such map that is independent of the first.

\begin{proposition}[Castelnuovo--Severi inequality]
\label{P:Castelnuovo-Severi}
Let $F$, $F_1$, $F_2$ be function fields of curves over $k$,
of genera $g$, $g_1$, $g_2$, respectively.
Suppose that $F_i \subseteq F$ for $i=1,2$
and the compositum of $F_1$ and $F_2$ in $F$ equals $F$.
Let $d_i=[F:F_i]$ for $i=1,2$.
Then
\[
	g \le d_1 g_1 + d_2 g_2 + (d_1-1)(d_2-1).
\]
\end{proposition}

\begin{proof}
See \cite{Stichtenoth1993}*{III.10.3}.
\end{proof}

\begin{proof}[Proof of Theorem~\ref{T:X_0}]
Let $X_0(n) \stackrel{h}{\To} \PP^1$ be a dominant rational map 
of minimal degree.

\emph{Case I: $h$ factors through $X_0(n) \stackrel{c}\To \PP^1$.}
Then
\[
	\deg h \ge D_0(n) = (1+o(1)) \frac{d^n}{n}
\]
by Proposition~\ref{P:properties of X_0}\eqref{I:D_0(n)},
so $\deg h$ is much larger than $n$ when $n$ is large.

\emph{Case II: $h$ does not factor through $X_0(n) \stackrel{c}\To \PP^1$.}
Then the compositum of $k(c)$ and $k(h)$ in the function field 
$k(X_0(n))$ is strictly larger than $k(c)$.
Because of the Galois group 
(Proposition~\ref{P:properties of X_0}\eqref{I:Galois group of X_0}),
the only nontrivial extension of $k(c)$ in $k(X_0(n))$
is the whole field $k(X_0(n))$.
Thus $k(c)$ and $k(h)$ generate $k(X_0(n))$.
By Proposition~\ref{P:Castelnuovo-Severi},
\[
	g(X_0(n)) \le (D_0(n)-1)(\deg h - 1).
\]
Thus
\[
	\deg h \ge 1 + \frac{g(X_0(n))}{D_0(n)-1} = \left(\frac{1}{2} - \frac{1}{2d} - o(1) \right) n
\]
as $n \to \infty$, 
by Proposition~\ref{P:properties of X_0}(\ref{I:D_0(n)},\ref{I:genus of X_0}).
\end{proof}

The following lemma, 
which says in particular that $Y_1(m,n)^{\zeta} \stackrel{z}\to \Aff^1$ 
is \'etale above $0$, 
will yield genus inequalities 
to combine with the Castelnuovo--Severi inequality
in the proof of Theorem~\ref{T:gonality}.

\begin{lemma}\label{L:zero preperiodic}
Let $k$ be a field of characteristic~$0$ such that $\mu_d \subseteq k$.
Let $m$ and $n$ be positive integers, and let $\zeta \in \mu_d - \{1\}$. 
Then the polynomial $\Phi_n(\zeta^{-1} f^{m-1}(0),c) \in k[c]$ 
has only simple roots,
and their number is $d^{m-2} D_1(n)$ if $m \ge 2$.
\end{lemma}

\begin{proof}
First suppose that $n$ does not divide $m - 1$. In this case, the roots of $\Phi_{m,n}(0,c)$ are distinct by \cite{HutzTowsley2015}*{Theorem 1.1}, and therefore the roots of $\Phi_n(\zeta^{-1}f^{m-1}(0),c)$ are distinct by \eqref{E:factorization}. 

Now suppose that $n$ divides $m - 1$. 
By \cite{Epstein2012}*{Proposition A.1}, the roots of $\Phi_n(0,c)$ are simple.
If $c$ is such a root, then the polynomial $f=f_c$ satisfies $f^n(0)=0$,
so $f^{m-1}(0)=0$ and $\Phi_n(\zeta^{-1}f^{m-1}(0),c) = 0$. 
Thus $\Phi_n(0,c)$ divides $\Phi_n(\zeta^{-1}f^{m-1}(0),c)$.
The factorization~\eqref{E:factorization} yields
\begin{equation}
\label{E:factorization 2}
	\frac{\Phi_{m,n}(0,c)}{\Phi_n(0,c)^{d-1}} = \prod_{\zeta \in \mu_d - \{1\}} \frac{\Phi_n(\zeta^{-1}f^{m-1}(0),c)}{\Phi_n(0,c)}.
\end{equation}
By \cite{HutzTowsley2015}*{Theorem 1.1}, 
$\Phi_{m,n}(0,c)/\Phi_n(0,c)^{d-1}$ has only simple roots, 
none of which are also roots of $\Phi_n(0,c)$. 
Combining this with \eqref{E:factorization 2} shows that 
$\Phi_n(\zeta^{-1}f^{m-1}(0),c)$ has only simple roots.

It remains to prove $\deg \Phi_n(\zeta^{-1}f^{m-1}(0),c) = d^{m-2} D_1(n)$.
In fact, this is \cite{Gao2016}*{Lemma~4.8}.
In our notation, the argument is as follows.
By induction on $m$, the degree of the polynomial $f^{m-1}(0) \in k[c]$ 
is $d^{m-2}$ if $m \ge 2$.
Hence, by induction on $e$, 
we have $\deg f^e(\zeta^{-1}f^{m-1}(0)) = d^e d^{m-2}$ for each $e \ge 0$.
Thus the $c$-degree of $f^e(\zeta^{-1}f^{m-1}(0))-\zeta^{-1}f^{m-1}(0)$
is $d^{m-2}$ times the $z$-degree of $f^e(z)-z$ for each $e \ge 1$.
Substituting $\zeta^{-1}f^{m-1}(0)$ for $z$ in \eqref{E:definition of Phi}
shows that 
\[
	\deg_c \Phi_n(\zeta^{-1}f^{m-1}(0),c) = d^{m-2} \deg_z \Phi_n(z,c) = d^{m-2} D_1(n).\qedhere
\]
\end{proof}

\begin{proof}[Proof of Theorem~\ref{T:gonality} in characteristic~$0$]
\hfill
\begin{enumerate}[\upshape (a)]
\item By Theorem~\ref{T:classification of dynatomic curves},
the dynatomic curves over $k$ are the curves $Y_1(n)$ and $Y_1(m,n)^\zeta$,
and they are geometrically irreducible.
\item The curves $Y_1(n)$ and $Y_1(m,n)^\zeta$ dominate $X_0(n)$,
so their gonalities are at least the gonality of $X_0(n)$,
by \cite{Poonen2007-gonality}*{Proposition~A.1(vii)}.
In light of Theorem~\ref{T:X_0},
it remains to prove that in each tower
\[
	\cdots \To Y_1(m,n)^\zeta \To Y_1(m-1,n)^\zeta \To \cdots \To Y_1(1,n)^\zeta
\]
for fixed $n$ and $\zeta$, the gonality tends to $\infty$ as $m \to \infty$.

Let $g_m = g(Y_1(m,n)^\zeta)$ and $\gamma_m = \gamma(Y_1(m,n)^\zeta)$.
Let $Y_1(m,n)^\zeta \stackrel{h}\To \PP^1$
be a dominant rational map of the minimal degree $\gamma_m$.

\emph{Case I: $h$ factors through a curve $Z$ with
$k(Y_1(m-1,n)^\zeta) \subseteq k(Z) \subsetneq k(Y_1(m,n)^\zeta)$.}
These inclusions imply $\gamma_{m-1} \le \gamma(Z)$
(by \cite{Poonen2007-gonality}*{Proposition~A.1(vii)})
and $2 \gamma(Z) \le \deg h = \gamma_m$, 
respectively.
Combining these yields $\gamma_m \ge 2 \gamma_{m-1}$.
Thus $\gamma_m$ grows, by induction on $m$.

\emph{Case II: $h$ does not factor through any such curve $Z$.}
Denote by $\pi_m$ the degree $d$ map
	\begin{align*}
		\pi_m \colon Y_1(m,n)^{\zeta} &\To Y_1(m-1,n)^{\zeta}\\
			(z,c) &\longmapsto (f(z),c),
	\end{align*}
and let $\widetilde{\pi}_m \colon X_1(m,n)^{\zeta} \to X_1(m-1,n)^{\zeta}$
be its extension to the smooth projective models.
Let $R_m$ be the ramification divisor of $\widetilde{\pi}_m$.

Applying Proposition~\ref{P:Castelnuovo-Severi} 
to $\pi_m$ and $h$ yields
\[
	g_m \le d g_{m-1} + (d-1)(\gamma_m-1),
\]
so it suffices to show that $g_m - dg_{m-1} \to \infty$ as $m \to \infty$. By Riemann--Hurwitz, this is equivalent to showing that $\deg R_m \to \infty$ as $m \to \infty$. 
In the fiber product diagram
\[
\xymatrix{
Y_1(m,n)^{\zeta} \ar[r]^-{\pi_m} \ar[d]_-z & Y_1(m-1,n)^{\zeta} \ar[d]^-z \\
\Aff^1 \ar[r]^-f & \Aff^1, \\
}
\]
both vertical morphisms are \'etale above $0$ 
by Lemma~\ref{L:zero preperiodic},
while $f$ has ramification index $d$ at $0$,
so $\widetilde{\pi}^m$ has ramification index $d$
at each point of $Y_1(m,n)^{\zeta}$ where $z=0$.
For $m \ge 2$, the number of such points is $d^{m-2} D_1(n)$
by Lemma~\ref{L:zero preperiodic}.
Thus $\deg R_m \ge (d-1) d^{m-2} D_1(n)$,
which tends to $\infty$ as $m \to \infty$.
\qedhere
\end{enumerate}
\end{proof}

\section{Gonality in characteristic \texorpdfstring{$p$}{p}}
\label{S:gonality in characteristic p}

\subsection{Reduction to the case of a finite field}
\label{S:reduction}

Let $\F_q=\F_p(\mu_d)$.

\begin{lemma}
\label{L:reduction}
Theorem~\ref{T:gonality} for $\F_q$ 
implies Theorem~\ref{T:gonality} for any
characteristic~$p$ field $k$ containing $\mu_d$.
\end{lemma}

\begin{proof}
Theorem~\ref{T:gonality}\eqref{I:geometrically irreducible}  for $\F_q$ 
implies that the dynatomic curves over $k$
are just the base extensions of the dynatomic curves $X$ over $\F_q$,
and that they are geometrically irreducible too.
We will prove in Section~\ref{S:proof of gonality in char p} 
that the smooth projective model of each $X$ has an $\F_q$-point.
Then Theorem~2.5(iii) and Proposition~A.1(ii) of~\cite{Poonen2007-gonality}
imply $\gamma(X_k) \ge \sqrt{\gamma(X)}$, which implies
Theorem~\ref{T:gonality}\eqref{I:gonality of dynatomic curves} for $k$.
\end{proof}

\subsection{Symbolic dynamics}
\label{S:symbolic dynamics}

In order to prove the finiteness of the set of dynatomic curves
of bounded degree over the $c$-line,
and to prove that the dynatomic curves of higher degree
have many $\F_q$-points above $c=\infty$,
we need to analyze the splitting above $c=\infty$ in dynatomic curves.
This splitting is given in Lemma~\ref{L:symbolic dynamics} below.
The symbolic dynamics approach yielding Lemma~\ref{L:symbolic dynamics} 
is standard (cf.~\cite{Morton1996}*{Lemma 1}),
but the ways that Lemma~\ref{L:symbolic dynamics} 
will be applied in subsequent sections are not standard.

View $f(z)$ as a polynomial in $z$ over the local field $\F_q((c^{-1}))$.
Normalize the valuation $v$ on $\F_q((c^{-1}))$ so that $v(c^{-1})=1$,
and extend $v$ to an algebraic closure.
Let $t=c^{-1/d}$, 
so $\F_q((t))$ is a degree~$d$ totally tamely ramified extension 
of $\F_q((c^{-1}))$.

\begin{lemma}
\label{L:symbolic dynamics}
\hfill
\begin{enumerate}[\upshape (a)]
\item \label{I:separable}
For any nonnegative integers $n>m$,
the polynomial $f^n(z)-f^m(z)$ over $\F_q(c)$ is separable
and splits completely over $\F_q((t))$.
\item \label{I:valuation}
Each zero of $f^n(z)-f^m(z)$ has valuation $-1/d$
and generates $\F_q((t))$ over $\F_q((c^{-1}))$.
\end{enumerate}
\end{lemma}

\begin{proof}
The ideas in the following argument are well known; 
cf.~\cite{Morton1996}*{Lemma 1}.
\begin{enumerate}[\upshape (a)]
\item 
For each $d$th root of $-c$, 
interpreting $(-c+z)^{1/d}$
as $(-c)^{1/d}(1- c^{-1} z)^{1/d}$ and expanding $(1- c^{-1} z)^{1/d}$
in a binomial series
defines a branch of the inverse of $z^d+c$
on the open disk $D \colonequals \{z \in \F_q((t)) : v(z) > v(c)\}$.
Taking the derivative of $(-c+z)^{1/d}$
shows that each branch is a contracting map $D \to D$.
These branches have disjoint images, each a smaller open disk around
a different $d$th root of $-c$.
Let $S$ be the set of these $d$ functions.
Each finite sequence of elements of $S$
defines a composition of functions,
and for each $m$, the images of the different $m$-fold compositions 
are disjoint open disks.
For each infinite sequence $s_1,s_2,\cdots$ of elements of $S$,
the images of $s_1 \cdots s_m$ for $m \ge 1$
are nested open disks whose radii tend to $0$,
so they have a unique point in their intersection;
denote it $[s_1 s_2 \cdots]$.
Any two distinct infinite sequences yield two points in disjoint disks,
so these points are distinct.
Since $f \circ s_1$ is the identity,
$f$ maps $[s_1 s_2 \cdots]$ to $[s_2 s_3 \cdots]$.
For fixed nonnegative integers $n > m$, 
any $n$-long sequence $s_1,\ldots,s_n$ in $S$ extends uniquely
to an infinite sequence $(s_i)$ satisfying $s_{i+n}=s_{i+m}$ for all $i \ge 1$,
and then $[s_1 s_2 \cdots]$ is a zero of $f^n(z)-f^m(z)$.
There are $d^n$ of these, so they are \emph{all} the zeros.
In particular, these zeros are distinct elements of $\F_q((t))$.
This implies that $f^n(z)-f^m(z)$ is separable.
\item 
The image of each $s \in S$ consists of elements of valuation exactly $-1/d$.
Thus each element $[s_1 s_2 \cdots]$ has valuation $-1/d$.
In particular, each zero of $f^n(z)-f^m(z)$ 
generates an extension field of $\F_q((c^{-1}))$ 
of ramification index divisible by $d$; 
this extension field can only be the whole field $\F_q((t))$.\qedhere
\end{enumerate}
\end{proof}

\subsection{Dynatomic curves of low degree}

It is here that we use the trick requiring the ground field to be finite.

\begin{lemma}
\label{L:low degree factors}
For each $e \ge 1$, 
the set of dynatomic curves $X$ over $\F_q$
such that ${\deg(X \stackrel{c}\to \PP^1)} = e$ is finite.
\end{lemma}

\begin{proof}
Suppose that $Q(z) = \sum_{r=0}^e q_r z^{e-r}$ 
is a monic degree~$e$ factor of $f^n(z)-f^m(z)$ over $\F_q(c)$
for some $n$ and $m$.
For each $r$, the coefficient $q_r$
is the $r$th elementary symmetric polynomial evaluated at the negatives
of the zeros of $Q$; those zeros have valuation $-1/d$
by Lemma~\ref{L:symbolic dynamics}\eqref{I:valuation}, 
so $v(q_r) \ge -r/d$.
On the other hand, by Gauss's lemma, $q_r \in \F_q[c]$,
so $\deg q_r \le r/d$.
Thus there are only finitely many possibilities for each $q_r$,
and hence finitely many possibilities for $Q$,
each of which yields one dynatomic curve.
\end{proof}

\subsection{Gonality of dynatomic curves}
\label{S:proof of gonality in char p} 

\begin{proof}[Proof of Theorem~\ref{T:gonality}]
By Lemma~\ref{L:reduction}, we may assume that $k=\F_q$.
\begin{enumerate}[\upshape (a)]
\item 
Let $X$ be the smooth projective model of a dynatomic curve, 
corresponding to a factor of $f^n(z)-f^m(z)$ for some $n$ and $m$.
By Lemma~\ref{L:symbolic dynamics}\eqref{I:valuation}, 
$f^n(z)-f^m(z)$ splits completely over $\F_q((t))$,
and the preimage of $\infty$ under $X \stackrel{c}\to \PP^1$
consists of $\F_q$-points, each of ramification index $d$.
Every irreducible component $Z$ of $X_{\Fbar_q}$ dominates $\PP^1$ via $c$
and hence must contain one of those $\F_q$-points, say $x$.
Then each $\Gal(\Fbar_q/\F_q)$-conjugate of $Z$ contains $x$.
On the other hand, since $X_{\Fbar_q}$ is smooth, 
its irreducible components are disjoint.
Thus $Z$ equals each of its conjugates, 
so $Z$ descends to an irreducible component of $X$,
which must be $X$ itself.
(This proves also that $X$ has an $\F_q$-point,
as promised in the proof of Lemma~\ref{L:reduction}.)
\item 
To bound gonality from below, 
we use Ogg's method of counting points over a finite field;
cf.~\cite{Ogg1974}, \cite{Poonen2007-gonality}*{Section~3}, 
and \cite{Doyle-et-al-preprint}*{Section~4}.
Let $X$ be the smooth projective model of a dynatomic curve.
Let $e = \deg(X \stackrel{c}\to \PP^1)$.
Each preimage of $\infty$ in $X$ is an $\F_q$-point of ramification index~$d$,
so there are $e/d$ such points.
On the other hand, $\PP^1$ has only $q+1$ points over $\F_q$,
so any nonconstant morphism $X \to \PP^1$
has degree at least $e/(d(q+1))$.
As $X$ varies, $e \to \infty$ by Lemma~\ref{L:low degree factors}.\qedhere
\end{enumerate}
\end{proof}

\section{Strong uniform boundedness of preperiodic points}
\label{S:uniform boundedness}

\begin{proof}[Proof of Theorem~\ref{T:uniform boundedness}]
Without loss of generality, $K=k(u)$ for some indeterminate $u$.
Given $L$, let $Y$ be the smooth projective integral curve over $k$
with function field $L$. 
The condition $[L:K] \le D$ implies that $Y$ has gonality at most $D$,
so each irreducible component of $Y_{\kbar}$ has gonality at most $D$.
For $c \in L$ not algebraic over $k$, 
if $z \in L$ and $n>m$ satisfy $f^n(z)-f^m(z)=0$,
then $(z,c)$ is a nonconstant and hence smooth $L$-point
of the curve $f^n(z)-f^m(z)=0$ in $\Aff^2$,
so it yields an $L$-point on 
a dynatomic curve $X$ corresponding to a factor of $f^n(z)-f^m(z)$.
This $L$-point defines a nonconstant $k$-morphism $Y \to X$,
so $\gamma(X) \le \gamma(Y) \le D$.
By Theorem~\ref{T:gonality}\eqref{I:gonality of dynatomic curves}, 
this places a uniform bound on $n$.
For each $n$, the number of preperiodic points of $z^d+c$ 
corresponding to that value of $n$
is uniformly bounded by $\sum_{m=0}^{n-1} \deg(f^n(z)-f^m(z)) = n d^n$,
so bounding $n$ bounds the number of preperiodic points too.

If $k$ is finite 
and $c$ lies in the maximal algebraic extension $\ell$ of $k$ in $L$,
then all the preperiodic points of $x^d+c$ in $L$ are in $\ell$,
but $[\ell:k] \le D$, so the number of preperiodic points
is uniformly bounded by $(\#k)^D$.
\end{proof}

To prove Theorem~\ref{T:number fields}, 
we need the following result of Frey \cite{Frey1994}*{Proposition 2}, 
which in turn relies on Faltings's theorems on rational points on 
subvarieties of abelian varieties.

\begin{lemma}\label{L:frey}
Let $C$ be a curve defined over a number field $K$.
Let $D \ge 1$. 
If there are infinitely many points $P \in C(\Kbar)$ 
of degree $\le D$ over $K$, 
then $\gamma(C) \le 2D$.
\end{lemma}

\begin{proof}[Proof of Theorem~\ref{T:number fields}]
For each $(m,n)$, let $S_{m,n}$ be the set of $(z_0,c) \in \Qbar \times \Qbar$
such that $z_0$ and $c$ belong to some number field $K$ of degree $\le D$
and iteration of $z^d+c$ maps $z_0$ into a cycle of length exactly $n$ 
after exactly $m$ steps;
thus $S_{m,n} \subseteq Y_1(m,n)(\Qbar)$.
Suppose that the conclusion fails; then for some $n \le N$, 
there exist infinitely many $m \ge 1$ for which $S_{m,n}$ is nonempty.
Fix such an $n$.

Let $m_0 \ge 1$.
By choice of $n$, the disjoint union $\coprod_{m \ge m_0} S_{m,n}$ is infinite.
For each $m \ge m_0$, the $c$-coordinate map $Y_1(m,n) \to \Aff^1$
factors through $Y_1(m_0,n)$, so we obtain maps of sets 
\begin{equation}
\label{E:composition}
	\coprod_{m \ge m_0} S_{m,n} \To S_{m_0,n} \To \Qbar.
\end{equation}
On the other hand, by Northcott's theorem \cite{Northcott1950}*{Theorem~3}, 
for any given $c \in \Qbar$,
the set of preperiodic points of $x^d+c$ of degree $\le D$ over $\Q$ is finite;
thus each $c \in \Qbar$ has finite preimage 
under the composition~\eqref{E:composition}.
Hence the image of~\eqref{E:composition} is infinite,
so $S_{m_0,n}$ is infinite.
Thus some $\Q(\mu_d)$-irreducible component $Y_1(m_0,n)^\zeta$ of $Y_1(m_0,n)$
contains infinitely many points of degree $\le D$ over $\Q(\mu_d)$.
By Lemma~\ref{L:frey}, $\gamma(Y_1(m_0,n)^\zeta) \le 2D$.

The previous paragraph applies for every integer $m_0 \ge 1$, contradicting 
Theorem~\ref{T:gonality}\eqref{I:gonality of dynatomic curves}.
\end{proof}

%****************************************************************************

%****************************************************************************
\section*{Acknowledgements} 

We thank Joseph Gunther, Holly Krieger, Andrew Obus, Padmavathi Srinivasan, 
Isabel Vogt, and Robin Zhang for comments and discussions.


\begin{bibdiv}
\begin{biblist}

% \bibselect{big}

\bib{Bousch1992}{thesis}{
	author = {Bousch, Thierry},
	title = {Sur quelques probl\`{e}mes de dynamique holomorphe},
	school = {PhD thesis, Universit\'{e} de Paris-Sud, Centre d'Orsay},
	year = {1992},
}

%\bib{Doyle2016}{article}{
%   author={Doyle, John R.},
%   title={Preperiodic portraits for unicritical polynomials},
%   journal={Proc. Amer. Math. Soc.},
%   volume={144},
%   date={2016},
%   number={7},
%   pages={2885--2899},
%   issn={0002-9939},
%   review={\MR{3487222}},
%}

\bib{Doyle-et-al-preprint}{misc}{
  author={Doyle, John R.},
  author={Krieger, Holly},
  author={Obus, Andrew},
  author={Pries, Rachel},
  author={Rubinstein-Salzedo, Simon},
  author={West, Lloyd},
  title={Reduction of dynatomic curves},
  date={2017-03-29},
  note={Preprint, \texttt {arXiv:1703.04172v2}\phantom {i}},
}

\bib{Epstein2012}{article}{
   author={Epstein, Adam},
   title={Integrality and rigidity for postcritically finite polynomials},
   note={With an appendix by Epstein and Bjorn Poonen},
   journal={Bull. Lond. Math. Soc.},
   volume={44},
   date={2012},
   number={1},
   pages={39--46},
   issn={0024-6093},
   review={\MR{2881322}},
}

\bib{Frey1994}{article}{
   author={Frey, Gerhard},
   title={Curves with infinitely many points of fixed degree},
   journal={Israel J. Math.},
   volume={85},
   date={1994},
   number={1-3},
   pages={79--83},
   issn={0021-2172},
   review={\MR{1264340}},
}

\bib{Gao2016}{article}{
   author={Gao, Yan},
   title={Preperiodic dynatomic curves for $z\mapsto z^d+c$},
   journal={Fund. Math.},
   volume={233},
   date={2016},
   number={1},
   pages={37--69},
   issn={0016-2736},
   review={\MR{3460633}},
}

\bib{HutzTowsley2015}{article}{
   author={Hutz, Benjamin},
   author={Towsley, Adam},
   title={Misiurewicz points for polynomial maps and transversality},
   journal={New York J. Math.},
   volume={21},
   date={2015},
   pages={297--319},
   issn={1076-9803},
   review={\MR{3358544}},
}

\bib{Lau-Schleicher1994}{article}{
author = {Lau, Eike},
author = {Schleicher, Dierk},
journal = {SUNY Stony Brook Preprint},
title = {Internal addresses in the Mandelbrot set and irreducibility of polynomials},
volume = {19},
year = {1994}
}

\bib{Merel1996}{article}{
  author={Merel, Lo{\"{\i }}c},
  title={Bornes pour la torsion des courbes elliptiques sur les corps de nombres},
  language={French},
  journal={Invent. Math.},
  volume={124},
  date={1996},
  number={1-3},
  pages={437\ndash 449},
  issn={0020-9910},
  review={\MR {1369424 (96i:11057)}},
}

\bib{Morton1996}{article}{
   author={Morton, Patrick},
   title={On certain algebraic curves related to polynomial maps},
   journal={Compositio Math.},
   volume={103},
   date={1996},
   number={3},
   pages={319--350},
   issn={0010-437X},
   review={\MR{1414593}},
}

\bib{Morton1998}{article}{
   author={Morton, Patrick},
   title={Galois groups of periodic points},
   journal={J. Algebra},
   volume={201},
   date={1998},
   number={2},
   pages={401--428},
   issn={0021-8693},
   review={\MR{1612390}},
   doi={10.1006/jabr.1997.7304},
}

\bib{Morton-Silverman1994}{article}{
  author={Morton, Patrick},
  author={Silverman, Joseph H.},
  title={Rational periodic points of rational functions},
  journal={Internat. Math. Res. Notices},
  date={1994},
  number={2},
  pages={97--110},
  issn={1073-7928},
  review={\MR {1264933 (95b:11066)}},
  doi={10.1155/S1073792894000127},
}

\bib{Nguyen-Saito1996preprint}{article}{
    author={Nguyen, Khac Viet},
    author={Saito, Masa-Hiko},
     title={$d$-gonality of modular curves and bounding torsions},
      date={1996-03-29},
      note={Preprint, \texttt{arXiv:alg-geom/9603024}\phantom{i}},
}

\bib{Northcott1950}{article}{
   author={Northcott, D. G.},
   title={Periodic points on an algebraic variety},
   journal={Ann. of Math. (2)},
   volume={51},
   date={1950},
   pages={167--177},
   issn={0003-486X},
   review={\MR{0034607 (11,615c)}},
}

\bib{Ogg1974}{article}{
  author={Ogg, Andrew P.},
  title={Hyperelliptic modular curves},
  journal={Bull. Soc. Math. France},
  volume={102},
  date={1974},
  pages={449\ndash 462},
  review={\MR {0364259 (51 \#514)}},
}

\bib{Poonen2007-gonality}{article}{
  author={Poonen, Bjorn},
  title={Gonality of modular curves in characteristic $p$},
  journal={Math. Res. Lett.},
  volume={14},
  date={2007},
  number={4},
  pages={691--701},
  issn={1073-2780},
  review={\MR {2335995}},
  doi={10.4310/MRL.2007.v14.n4.a14},
}

\bib{Schleicher2017}{article}{
   author={Schleicher, Dierk},
   title={Internal addresses of the Mandelbrot set and Galois groups of
   polynomials},
   journal={Arnold Math. J.},
   volume={3},
   date={2017},
   number={1},
   pages={1--35},
   issn={2199-6792},
   review={\MR{3646529}},
}

\bib{Stichtenoth1993}{book}{
  author={Stichtenoth, Henning},
  title={Algebraic function fields and codes},
  series={Universitext},
  publisher={Springer-Verlag},
  place={Berlin},
  date={1993},
  pages={x+260},
  isbn={3-540-56489-6},
  review={\MR {1251961 (94k:14016)}},
}

\end{biblist}
\end{bibdiv}
\end{document}